\documentclass [11pt]{article}  
\usepackage{a4}                 
\usepackage{paralist}           
\usepackage{graphicx}
\usepackage{amssymb}                               
\usepackage{float}
\usepackage{amsmath}
\usepackage{psfrag}

\usepackage{amsmath,amsthm}

\theoremstyle{plain}
\newtheorem{thm}{Theorem}[section]
\newtheorem*{thmspec}{Theorem 4.2}
\newtheorem*{thmspec2}{Corollary 4.4}
\newtheorem*{thmspec3}{Theorem 5.5}
\newtheorem*{thmspec4}{Theorem 6.6}
\newtheorem{lem}[thm]{Lemma}

\newtheorem{cor}[thm]{Corollary}
\newtheorem*{mquestion}{Motivating question}

\theoremstyle{definition}
\newtheorem{defn}[thm]{Definition}
\newtheorem{ques}{Question}

\theoremstyle{remark}
\newtheorem*{rem}{Remark}

\DeclareMathOperator{\rank}{rank}
\DeclareMathOperator{\ab}{ab}
\DeclareMathOperator{\torord}{TorOrd}

\title{\textbf{Finding non-trivial elements\\and splittings in groups}}
\author{Maurice Chiodo}

\date{}

\begin{document}

\maketitle

\begin{abstract}
It is well known that the triviality problem for finitely presented groups is unsolvable; we ask the question of whether there exists a general procedure to produce a non-trivial element from a finite presentation of a non-trivial group. If not, then this would resolve an open problem by J. Wiegold: `Is every finitely generated perfect group the normal closure of one element?' We prove a weakened version of our question: there is no general procedure to pick a non-trivial generator from a finite presentation of a non-trivial group. We also show there is neither a general procedure to decompose a finite presentation of a non-trivial free product into two non-trivial finitely presented factors, nor one to construct an embedding from one finitely presented group into another in which it embeds. We apply our results to show that a construction by Stallings on splitting groups with more than one end can never be made algorithmic, nor can the process of splitting connect sums of non-simply connected closed 4-manifolds.
\end{abstract}

\let\thefootnote\relax\footnotetext{2010 \textit{AMS Classification:} 20F10 (03D80).}
\let\thefootnote\relax
\let\thefootnote\relax\footnotetext{\textit{The author was supported by an} Australian Postgraduate Award.}

\section{Introduction}

Combinatorial group theory involves the study of groups given by presentations $G=\left \langle X|R \right \rangle$, where $X$ is a set of generators and $R$ a set of defining relations, both often taken to be finite. Yet despite this simplification many interesting algorithmic questions about such groups have been shown to be recursively unsolvable.  In particular there are groups $G$ given by finite presentations for which the word problem, of determining whether a word $w$ in the generators is trivial in $G$, is recursively unsolvable (\cite{Boone}). The isomorphism problem of deciding whether two presentations define isomorphic groups is also unsolvable. Even the triviality problem of deciding whether or not $G \cong \{1\}$ is unsolvable (\cite{Adian},\cite{Rabin}). 
However, if we consider certain classes of groups then these problems are sometimes more tractable. For example, in the class of finitely presented abelian groups, one can solve the word problem uniformly for all presentations, as well as the isomorphism problem. For an excellent survey on decision problems in group theory discussing these results and many others, see \cite{Mille-92}.

A good example to consider is the class of finitely presented non-trivial simple groups, which can be shown to have solvable word problem as follows: take a finite presentation $\left \langle X|R \right \rangle$=$\left \langle x_{1}, \ldots x_{n}|R \right \rangle$ of a non-trivial simple group and a fixed word $s$ representing a non-trivial element. Given an arbitrary word $w$ on $X$, begin an enumeration $w_{1}, w_{2}, \ldots$ of all trivial words in $\left \langle X|R \right \rangle$, and at the same time an enumeration $y_{1}, y_{2}, \ldots$ of all trivial words in $\left \langle X|R, w \right \rangle$
. Now look for $w$ in the first list, and $s$ in the second list. If $w$ is trivial then it will appear in the first list, if $w$ is non-trivial then $s$ will appear in the second list; it is clear that precisely one of these will occur as $\left \langle X|R \right \rangle$ defines a non-trivial simple group. Moreover, this algorithm can easily be made \emph{uniform} (that is, recursively constructible for each case) by searching for $w$ in the first list, and searching for all $x_{i}$ in the second list. If the first search halts then $w$ is trivial, if the second halts then $w$ is non-trivial; again, precisely one will halt.

Notice that our first algorithm works for each individual simple group because in any non-trivial group there exists a non-trivial element. But there is no reason to assume that the process of selecting a non-trivial element can be made uniform. The word problem mentions nothing about being able to recursively construct such an algorithm for the given presentation, only that one must exist. In fact there is no universal algorithm to solve the word problem on all groups with soluble word problem, nor is the class of finitely presented groups with soluble word problem recursively enumerable (see \cite{BooneRogers}).

By taking a little step back, we see that we have come across an interesting problem: Is there a general method that, when given ANY non-trivial group, produces a single non-trivial element? We state this formally below; it will be the motivation behind all our further work in this paper.

\begin{mquestion}
Is there a partial algorithm on finite presentations of groups that, on input of a finite presentation $P$ of a non-trivial group, outputs a word $w$ on the generators of $P$ such that $w$ is non-trivial in that group?
\end{mquestion}

One would presume a sensible place to start looking would be the finite generating set, naively hoping to sift out one of the generators as non-trivial (like panning for gold). We shall later show that this approach is impossible in general. Note that we are only asking for a partial algorithm, and are only concerned with its behaviour when given a non-trivial group. If such an algorithm was to be given a presentation of the trivial group, it may never halt, or my halt and output some word $w$ which would of course be trivial. Therefore such an algorithm can't be used in the `obvious' way to enumerate all finite presentations of non-trivial groups, a set known to be not recursively enumerable (\cite{Rogers}).

Our question is easily reduced to the recursive class of finitely presented perfect groups. For if our group $G$ is not perfect then we can recursively form the abelianisation and identify a non-trivial generator there, which is possible since the word problem is uniformly solvable for finitely presented abelian groups. Any preimage of this generator will therefore be non-trivial in $G$. So from hereon we shall usually assume that we have already recursively tested our group to be perfect. Notice that if $G$ happens to be presented with 2 generators, and is perfect and non-trivial, then neither generator can be trivial. So our question is fully resolved for non-trivial two generator groups.

As the world of mathematics is heavily intertwined, it is not surprising to find connections to existing open problems. One such example is the following open problem in group theory, posed by J. Wiegold as 5.52 in \cite{notebook}: `Is every finitely generated perfect group the normal closure of one element?' If the answer to this is yes, then we would have an algorithm that satisfies our motivating question as follows: With $G$ assumed to be perfect, we recursively enumerate all presentations with the same generating set as $G$ and precisely one extra relator (some $w_{i}$). Assuming Wiegold's question has a positive answer, at least one such presentation will describe the trivial group, which we can recursively search for and will eventually find. The corresponding $w_{i}$ is then non-trivial in $G$. We note here the author believes that the answer to Wiegold's question and to the motivating question is `no' in each case.

Looking back at our motivating question again, one may ask why a proof that the triviality problem is undecidable doesn't extend to show that there is no algorithm to extract a non-trivial element. The known proofs involve the Adian-Rabin theorem (\cite{Adian}, \cite{Rabin}). That is, there is a uniform procedure that, on input of a finite presentation $P$ of a group and a word $w$ on the generators, recursively constructs a finite presentation $P(w)$ such that this new group is trivial if and only if $w$ is trivial in $P$, and moreover whenever $w$ is non-trivial then $P$ embeds in $P(w)$. All this is done without prior knowledge of whether $w$ is trivial in $P$ or not. This reduces the triviality problem to the word problem (thus showing it is undecidable). However, by construction, the group $P(w)$ is known to be 2 generator and perfect, so as mentioned previously neither generator will be trivial. So we have such an algorithm for these $P(w)$ groups whenever they are non-trivial. But this does not suffice, as we want an algorithm for all finitely presented non-trivial groups.

We are quickly led to modify our main question, by replacing `non-trivial' with other group properties. For example, knowing our presentation defines a non-abelian group, can we extract two elements that do not commute? Or, knowing our presentation defines a group that is the free product of two non-trivial groups, can we re-write it as a free product with non-trivial factors? Similarly, knowing that a finite presentation $P$ defines a group that embeds into the group defined by another finite presentation $Q$, can we construct such an embedding?  It is known that the properties of being trivial, or abelian, or a free product, or a subgroup, are not algorithmically recognisable for finitely presented groups (\cite{Mille-92}). But is the knowledge that a group lies outside one such class sufficient to algorithmically demonstrate it? These particular types of decision problems are interesting as we are dancing very close to problems known to be algorithmically impossible, but hoping that our extra piece of algorithmically indeterminable information may be of some use. We resolve some of these types of questions with the following main results of this paper:

\begin{thmspec}\label{first main}Fix any $k>0$. Then there is no algorithm that, on input of a finite presentation $P=\left \langle X|R \right \rangle$ of a non-trivial group $\overline{P}$, outputs a word $w$ on $X$ of length at most $k$ such that $w$ is non-trivial in $\overline{P}$ (If $P$ is a group presentation, then $\overline{P}$ denotes the group presented by $P$).
\end{thmspec}
Though this does not resolve our motivating question, it is still an enlightening result in its own right. We use it to deduce the following corollary.

\begin{thmspec2}\label{second main}Fix any $k>0$. Then there is no algorithm that, on input of a finite presentation $P=\left \langle X|R \right \rangle$ of a non-abelian group $\overline{P}$, outputs two words $w, w'$ on $X$, each of length at most $ k$, such that $[w, w']$ is non-trivial in $\overline{P}$.
\end{thmspec2}

The next two results are proved in full generality, without needing to weaken the questions. We use theorem \ref{split} later to show that a construction by Stallings on splitting groups with more than one end can never be made algorithmic, nor can the process of decomposing the connect sum of two non-simply connected closed 4-manifolds into non-simply connected summands.

\begin{thmspec3}\label{third main}There is no algorithm that, on input of a finite presentation $P=\left \langle X|R \right \rangle$ of a group that is a free product of two non-trivial finitely presented groups, outputs two finite presentations $P_{1}, P_{2}$ which present non-trivial groups and whose free product is isomorphic to $\overline{P}$.
\end{thmspec3}

\begin{thmspec4}\label{fourth main}
There is no algorithm that, on input of two finite presentations $P=\langle X|R\rangle$ and $Q=\langle Y |S \rangle$ such that $\overline{P}$ embeds in $\overline{Q}$, outputs an explicit map $\theta$ from $X$ to words in $Y$ such that $\theta$ extends to an embedding $\overline{\theta}: \overline{P} \hookrightarrow \overline{Q}$.
\end{thmspec4}

The main tool to be used in proving all the above results will be the following interesting application of a construction by Boone: 
there is an explicit recursive procedure that, on input of $m,n \in \mathbb{N}$, outputs a finite presentation $\Pi_{m,n}$ of a torsion free perfect group such that $\overline{\Pi}_{m,n} = \{1\} \textnormal{ if and only if }n \in W_{m}$, the $m^{th}$ recursively enumerable set. This is done without knowing \textit{a priori} whether $n$ lies in $W_{m}$ or not.

Using these $\Pi_{m,n}$ we are able to encode the following two ideas from recursion theory, providing the algorithmic complexity for our final group theoretic results. It helps to think of these as alternative recursion theoretic results to the standard fact that there is a recursively enumerable yet non-recursive set.

1. Given a recursively enumerable set $W_{n}$ and a finite set $F \nsubseteq W_{n}$, there is no general procedure to recursively find a proper subset $A\subset F$ such that $A \nsubseteq W_{n}$ (we use this to prove theorem \ref{unbounded word} and corollary \ref{ab bounded}).

2. Given a recursively enumerable set $W_{n}$ and a finite set $F$ such that $|F\cap W_{n}| \leq 1$, there is no general procedure to recursively find an element of $F$ not lying in $W_{n}$ (we use this to prove theorem \ref{split} and theorem \ref{no embeddings}).

We will eventually show, by taking various free products of some well-chosen $\Pi_{m,n}$, that the existence of any algorithm described in theorem \ref{unbounded word}, corollary \ref{ab bounded}, theorem \ref{split} or theorem \ref{no embeddings} would yield algorithms violating these recursion theory results.

\textbf{Notation:} We shall adopt the following notation and conventions for the remainder of this paper. When $P=\left\langle X|R \right\rangle$ is a (semi)group presentation, we denote by $\overline{P}$ the (semi)group presented by $P$. If Q is another finite presentation then we denote their free product presentation, given by taking the disjoint union of their generators and relators, by $P*Q$. If $\{x_{1}, \ldots, x_{n}\}$ is a finite set, then for any $1 \leq i \leq n$ we define $\{x_{1}, \ldots, \hat{x}_{i}, \ldots, x_{n}\} := \{x_{1}, \ldots, x_{n}\} \setminus \{x_{i}\} $. If $X$ is a set, then we denote by $X^{-1}$ a set of the same cardinality as $X$ (considered an `inverse' set to $X$), and $W(X)$ the finite words on $X \cup X^{-1}$, including the empty word. A word $w \in W(X)$ is said to be \emph{positive} if it contains no element of $X^{-1}$; denote the set of all such words by $\Omega(X)$. The \textit{length} $l(w)$ of a word $w \in W(X)$ is the number of symbols in $w$ before free reductions. For a finitely-generated group $G$ the \emph{Rank} of $G$, $\rank(G)$, is the minimal size of a generating set for $G$, and $G$ is \emph{indecomposable} if it is non-trivial yet cannot be written as the free product of two non-trivial groups. 


\textbf{Acknowledgements:} The author wishes to thank Chuck Miller for his many conversations and valuable insight, Rod Downey for his guidance with recursion theory, and Andrew Glass for his time in considering many of these problems. Thanks also go to Lawrence Reeves and Jack Button for their help with proof reading.

\section{Turing Machines and Group Constructions}

The following section is taken largely from \cite{Rot} chapter 12 which the reader may wish to familiarise himself with, we give a summary of the necessary ideas and results.

There are various definitions of Turing machines and partial recursive functions, all of which can be shown to be equivalent. See \cite{Rogers} chapter 1 for an excellent introduction to recursive function theory, especially $\S 1.6$. We shall employ the definition found in \cite{Rot}, and shall adhere to the following notation.

\begin{defn}\label{re func}
Let $T$ be a Turing machine with alphabet $S$. The \emph{halting set} $e(T)$ of $T$ is the set of tapes that $T$ eventually halts on, which can be viewed as a subset of $\Omega(S)$. We adopt the notation $T(w) \downarrow$ whenever $T(w)$ halts ($w \in e(T)$), and $T(w) \uparrow$ whenever $T(w)$ doesn't halt ($w \notin e(T)$).
We define the $m^{th}$ \emph{partial recursive function} $\varphi_{m}$ as follows: Take the $m^{th}$ Turing machine $T_{m}$ with alphabet $\{s_{0}, s_{1}\}$. Then $\varphi_{m}$ takes as input values in $\mathbb{N}$ via $\varphi_{m}(n):=T_{m}(s_{1}^{n+1}) $; $s_{1}^{n+1}$ being the tape with $n+1$ successive images of $s_{1}$ on it. The $m^{th}$ \emph{partial recursive set} $W_{m}$ is then defined as the domain of $\varphi_{m}$.
\end{defn}

We recall some important constructions and results in group theory by Post, Boone, Adian and Rabin.

\begin{thm}[Post, \cite{Rot} lemma $12.4$]\label{Post}
Let $T$ be a Turing machine with alphabet $S$. Then we can recursively construct a finite semigroup presentation denoted $\Gamma(T)$ such that, whenever $w \in \Omega(S)$, we have $hq_{1}wh=q$ in $\overline{\Gamma(T)}$ if and only if $T(w)$ halts (where $h,q,q_{1}$ are fixed generators from $\Gamma(T)$).
\end{thm}

Using this and the existence of a recursively enumerable but non-recursive set as given in \cite{Rogers} $\S 5.2$ theorem $VI$, Post (\cite{Post}) was able to show that there exists a finitely presented semigroup with insoluble word problem.

\begin{thm}[Boone, \cite{Rot} lemma $12.7$]\label{boone}
Let $T$ be a Turing machine with alphabet $S$. Then we can recursively construct a finite presentation of a group $B(T)=\left \langle X_{T}|R_{T} \right \rangle$, uniform in $T$. This has the property that, for any $w \in \Omega(S)$ we can form a word $\beta(w) \in W(X_{T})$, uniform in $w$, such that $\beta(w) =e$ in $\overline{B(T)}$ if and only if $T(w)$ halts.
\end{thm}

Though we do not provide a proof here, we note that $B(T)$ is built up of amalgamated products and HNN extensions, beginning with free groups. Since amalgamated products and HNN extensions preserve the property of being torsion free (see theorem  \ref{tor thm}), then $\overline{B(T)}$ must be torsion free.


Using theorem \ref{boone} and the existence of a recursively enumerable but non-recursive set, Boone (\cite{Boone}) was able to show that there exists a finitely presented group with insoluble word problem. Most of our efforts in this paper will involve encoding somewhat more elaborate Turing machines into generalisations of the construction $B(T)$ to obtain our desired results. We will also make use of the following important result, known as the Adian-Rabin theorem (\cite{Adian},\cite{Rabin}).

\begin{thm}[Adian-Rabin]\label{rabin}
There is a uniform construction that, for each finite presentation $P=\left \langle X|R \right \rangle$ of a group and each $w \in W(X)$, produces a finite presentation $P(w)$ of a group and an explicit homomorphism $\phi : \overline{P} \to \overline{P(w)}$ such that:
\\$1$. If $w \neq e$ in $\overline{P}$, then $\phi: \overline{P} \hookrightarrow \overline{P(w)}$ is an embedding.
\\$2$. If $w = e$ in $\overline{P}$, then $\overline{P(w)} = \{1\}$.
\end{thm}

Though we do not provide a full proof here, we explicitly state the presentation used. Note that this is neither Adian's nor Rabin's original construction, but a later version due to Gordon (\cite{Gor}). From $P=\left \langle X|R \right \rangle$ we construct the presentation $P(w)$ with generating set $X \cup \{a,b,c\}$, and relators $R$ along with:
\begin{align*}
a^{-1}ba\ =\  & c^{-1}b^{-1}cbc
\\ a^{-2}b^{-1}aba^{2}\ =\  & c^{-2}b^{-1}cbc^{2}
\\a^{-3}[w,b]a^{3}\ =\  & c^{-3}bc^{3}
\\a^{-(3+i)}x_{i}ba^{(3+i)}\ =\  & c^{-(3+i)}bc^{(3+i)} \ \forall x_{i} \in X
\end{align*}
We also note that in this proof $P(w)$ is built up from amalgamated products and HNN extensions, beginning with the presentation $P$. Thus, whenever $w$ is non-trivial, $\overline{P(w)}$ has torsion if and only if $\overline{P}$ has torsion (see theorem  \ref{tor thm}). From this presentation we can immediately observe the following corollary.

\begin{cor}\label{rabin lem}
Using the notation above, if $\overline{P(w)}$ is non-trivial then it is perfect and has rank precisely $2$, generated by $b$ and $ca^{-1}$.
\end{cor}

We combine the above results in the following way, as pointed out to the author by Chuck Miller: Given any pair $m,n \in \mathbb{N}$, take the $m^{th}$ Turing machine $T_{m}$ with alphabet $\{s_{0}, s_{1}\}$. Now use Boone's construction to form the group presentation $B(T_{m})$ and word $\beta(s_{1}^{n+1})$. Input this presentation/word pair into Gordon's construction for the Adian-Rabin theorem to form the presentation $(B(T_{m}))(\beta(s_{1}^{n+1}))$, which we shall denote as $\Pi_{m,n}$. Then by the properties of the Boone and Gordon constructions, we have that $n \in W_{m}$ if and only if $\beta(s^{n+1})= e$ in $\overline{B(T_{m})}$ (theorem \ref{boone}), if and only if $\overline{\Pi}_{m,n} = \{1\}$ (theorem \ref{rabin}). From the comments after theorems \ref{boone} and \ref{rabin}, we observe that if $\overline{\Pi}_{m,n}$ is non-trivial then it must be torsion free. So we have just proven the following:

\begin{thm}\label{chuck result}
There is an explicit recursive procedure that, on input of $m,n \in \mathbb{N}$, outputs a finite presentation $\Pi_{m,n}$ of a torsion free, perfect group such that $\overline{\Pi}_{m,n} = \{1\}$ if and only if $n \in W_{m}$.
\end{thm}

\section{Tools from recursion theory and decidability}

Recalling \emph{Cantor's pairing function} $\left\langle . , .\right\rangle : \mathbb{N}\times\mathbb{N} \to \mathbb{N}$, $\left\langle x,y\right\rangle := \frac{1}{2}(x+y)(x+y+1)+y$ which is a bijection from $\mathbb{N}\times\mathbb{N}$ to $\mathbb{N}$, one can extend this inductively to define a bijection from $\mathbb{N}^{n}$ to $\mathbb{N}$ by $\left\langle x_{1},  \ldots, x_{n} \right\rangle := \left\langle \left\langle x_{1}, \ldots, x_{n-1} \right\rangle ,x_{n} \right\rangle$. We note that this function, and all its extensions, are recursively computable. The following two results can be found in \cite{Rogers} as $\S$1.8 theorem $V$, and $\S$11.2 theorem $I$.

\begin{thm}[The s-m-n theorem, or substitution theorem]\label{smn}For all $m, n \in \mathbb{N}$, a partial function $f:\mathbb{N}^{m+n} \to \mathbb{N}$ is partial-recursive if and only if there is a recursive function $s:\mathbb{N}^{m} \to \mathbb{N}$ such that, for all $e_{1}, \ldots, e_{m}, x_{1}, \ldots, x_{n} \in \mathbb{N}$ we have that
\\$f(e_{1}, \ldots, e_{m}, x_{1}, \ldots, x_{n})= \varphi_{s(e_{1}, \ldots, e_{m})}(\left\langle x_{1}, \ldots, x_{n}\right\rangle )$.
\end{thm}

\begin{thm}[Kleene recursion theorem]\label{Kleene}Let $f:\mathbb{N}\to \mathbb{N}$ be a recursive function. Then there exists $n \in \mathbb{N}$ with $\varphi_{n}=\varphi_{f(n)}$.
\end{thm}

The following two lemmata were inspired by a correspondence between the author and Rod Downey, and will form the recursion theoretic basis of our main results in group theory.

\begin{lem}\label{arb length}
Fix any $k>0$. Then there is no partial recursive function $g: \mathbb{N}^{k+2} \to \mathbb{N}$ such that, given $n, x_{0}, \ldots, x_{k} \in \mathbb{N}$ satisfying $\{ x_{0}, \ldots, x_{k}\} \nsubseteq  W_{n}$, we have that $g(n, x_{0}, \ldots, x_{k})$ halts with output  $x_{i} \in \{x_{0}, \ldots, x_{k}\}$ such that $\{x_{0}, \ldots, \hat{x_{i}}, \ldots x_{k} \} \nsubseteq W_{n}$.
\\That is, given a recursively enumerable set $W_{n}$ and a finite set $F \nsubseteq W_{n}$, we can't, in general, recursively find a proper subset $A\subset F$ such that $A \nsubseteq W_{n}$.
\end{lem}

\begin{proof}
Assume such a $g$ exists. Define $f: \mathbb{N}\times \mathbb{N} \to \mathbb{N}$ by 
\[ f(n,m):= \left \{
 \begin{array}{l} 
    0 \textnormal{ if } g(n,0, \ldots, k)=j \in \{0, \ldots, k\}  \textnormal{ and } m\in \{0, \ldots, \hat{j}, \ldots, k\} \\
    \uparrow \textnormal{ in all other cases} \\
 \end{array}
\right.
\]
Then $f$ is partial-recursive, since $g$ is.
By theorem \ref{smn}, there exists a recursive function $s: \mathbb{N} \to \mathbb{N}$ such that $f(n,m)=\varphi_{s(n)}(m)$ for all $m, n$.
Since $s$ is recursive, theorem \ref{Kleene} shows that there must be some $n'$ such that $\varphi_{s(n')}=\varphi_{n'}$. Thus $f(n', m) = \varphi_{n'}(m) $ for all $m \in \mathbb{N}$.
Moreover, by definition, $\varphi_{n'}(m)$ can halt on at most $k$ of the $k+1$ cases $m=0, \ldots, m=k$ (if at all), and no other values.
Thus $|W_{n'}| \leq k$ since $W_{n'}$ is precisely the halting set of $\varphi_{n'}$.
So we must have that at least one of $\{0, \ldots, k\}$ does not lie in $W_{n'}$.
Thus $g(n', 0 , \ldots, k)$ will halt and output some $j$ in $\{0, \ldots, k\}$ such that  $\{0, \ldots, \hat{j},\ldots, k\} \nsubseteq  W_{n'}$ (by construction of $g$). 
But since $g(n', 0 , \ldots, k)$ halts and outputs $j$ in $\{0, \ldots, k\}$, then $f(n',m)$ will halt for all $m \in \{0, \ldots, \hat{j}, \ldots, k\}$ (by construction of $f$). Hence $\varphi_{n'}(m)$ will halt for all $m \in \{0, \ldots, \hat{j}, \ldots, k\}$ (by definition of $\varphi_{n'}$), and so $\{0, \ldots, \hat{j}, \ldots, k\} \subseteq  W_{n'}$ since $W_{n'}$ is precisely the halting set of $\varphi_{n'}$. Thus we have a contradiction, as we showed $\{0, \ldots, \hat{j},\ldots, k\} \nsubseteq  W_{n'}$, so no such $g$ can exist.
\end{proof}

\begin{lem}\label{two groups recursion}
Fix any $k>0$. Then there is no partial recursive function $g:\mathbb{N}^{k+2} \to \mathbb{N}$ such that, given $n, x_{0}, \ldots, x_{k} \in \mathbb{N}$ satisfying $|\{x_{0}, \ldots, x_{k}\} \cap W_{n}| \leq 1$, we have that $g(n,x_{0}, \ldots, x_{k})$ halts with output $x_{i} \in \{x_{0}, \ldots, x_{k}\}$ such that $x_{i} \notin  W_{n}$.
\\That is, given a recursively enumerable set $W_{n}$ and a finite set $F$ such that $|F\cap W_{n}| \leq 1$, we can't, in general, recursively find an element of $F$ not lying in $W_{n}$.
\end{lem}

\begin{proof}
Assume such a $g$ exists. Define $f: \mathbb{N}\times \mathbb{N} \to \mathbb{N}$ by 
\[ f(n,m):= \left \{
 \begin{array}{l} 
    0 \textnormal{ if } g(n,0, \ldots, k)=j \in \{0, \ldots, k\} \textnormal{ and } m=j \\
        \uparrow \textnormal{ in all other cases} \\
 \end{array}
\right.
\]
Then $f$ is partial-recursive, since $g$ is.
By theorem \ref{smn}, there exists a recursive function $s: \mathbb{N} \to \mathbb{N}$ such that $f(n,m)=\varphi_{s(n)}(m)$ for all $m, n$. 
Since $s$ is recursive, theorem \ref{Kleene} shows that there must be some $n'$ such that $\varphi_{s(n')}=\varphi_{n'}$. Thus $f(n', m) = \varphi_{n'}(m) $ for all $m \in \mathbb{N}$. Moreover, by definition, $\varphi_{n'}(m)$ can halt on at most one of the cases $m=0,\ldots, m= k$ (if at all), and no other values. Thus $|\{0, \ldots, k\} \cap W_{n'}| \leq 1$. So $g(n',0, \ldots, k)$ will halt and output $j$ in $\{0, \ldots, k\}$ but not in $W_{n'}$ (by construction of $g$). 
But since $g(n', 0, \ldots, k)$ halts with output $j$ in $\{0, \ldots, k\}$, then $f(n',j)$ halts (by construction of $f$). Hence $\varphi_{n'}(j)$ halts (by definition of $\varphi_{n'}$), and so $j \in W_{n'}$ since $W_{n'}$ is precisely the halting set of $\varphi_{n'}$. Thus we have a contradiction, as we showed $j \notin W_{n'}$, so no such $g$ can exist.
\end{proof}

We require the following fundamental results in group computability, made clear in \cite{Mille-92}, which we shall appeal to in our final proofs.

\begin{lem}
Given a finite group presentation $P=\left \langle X|R \right \rangle$, we can recursively form the abelianisation presentation $P^{\ab}$ by adding the relations $[x_{i}, x_{j}]$ for all $x_{i}, x_{j}\in X$. Then $\overline{P^{\ab}}$ is the abelianisation of the group $\overline{P}$. Moreover, the word and isomorphism problems are uniformly solvable for finitely presented abelian groups.
\end{lem}

\begin{lem}\label{isomorphism}
There is an algorithm that, on input of two finite presentations of groups $P,Q$, halts if and only if $\overline{P} \cong \overline{Q}$, and outputs an explicit isomorphism between them as a map on the generators. Hence, given a finite group presentation $P$, the set of all finite presentations $P_{i}$ satisfying $\overline{P_{i}} \cong \overline{P}$ is recursively enumerable.
\end{lem}

\section{Finding non-trivial or non-commuting elements}

We are now able to prove the first of the main results of this paper on finding non-trivial or non-commuting elements.

\begin{lem}\label{no n groups}
Fix any $k>0$. Then there is no algorithm that, on input of $k+1$ finite presentations of groups such that at least one describes a non-trivial group, removes one so that of the remaining presentations at least one describes a non-trivial group.
\end{lem}

\begin{proof}
Assume such an algorithm exists; we use this to contradict lemma \ref{arb length}. Given any $(k+2)$-tuple $(n,x_{0}, \ldots, x_{k})$ with $\{x_{0}, \ldots, x_{k}  \} \nsubseteq W_{n}$, form the groups with presentations $\Pi_{n,x_{0}}, \ldots,  \Pi_{n,x_{k}}$ as in theorem \ref{chuck result} whereby $\overline{\Pi}_{n,x_{i}} =\{e\}$ if and only if $x_{i} \in W_{n}$. 
Hence at least one of $\overline{\Pi}_{n,x_{0}},\ldots,  \overline{\Pi}_{n,x_{k}} \neq \{e\}$. Now algorithmically remove one of these presentations such that those remaining do not all define trivial groups. This immediately yields one of $x_{0}, \ldots, x_{k}$ to remove such that not all remaining $x_{i}$ lie in $W_{n}$. But lemma \ref{arb length} asserts that no such algorithm can exist.
\end{proof}

\begin{thm}\label{unbounded word}
Fix any $k>0$. Then there is no algorithm that, on input of a finite presentation $P=\left \langle X|R \right \rangle$ of a non-trivial group $\overline{P}$, outputs a word $w \in W(X)$ of length at most $k$ such that $w$ is non-trivial in $\overline{P}$.
\end{thm}

\begin{proof}
Assume such an algorithm exists; we use this to contradict lemma \ref{no n groups}. Given any $k+1$ finite presentations of groups $P_{0}=\left \langle X_{0}|R_{0} \right \rangle, \ldots,$ $P_{k}=\left \langle X_{k}|R_{k} \right \rangle$ for which at least one describes a non-trivial group, form their free product with presentation $Q=\left \langle X_{0}, \ldots, X_{k}|R_{0}, \ldots, R_{k} \right \rangle$. Apply the algorithm to $Q$ to give a non-trivial word $w$ of length at most $k$. Then there is an $i$ such that no element of $X_i$ appears in $w$. So removing $P_{i}$ from our original collection will not leave only presentations of the trivial group. But theorem \ref{no n groups} asserts no such algorithm can exist.
\end{proof}

\begin{cor}\label{pick a gen}
There is no algorithm that, on input of a finite presentation $P=\left \langle X|R \right \rangle$ of a non-trivial group $\overline{P}$, outputs a generator $x$ in $X$ such that $x$ is non-trivial in $\overline{P}$.
\end{cor}

It is tempting to try and generalise the above results, by removing the bound on the length of output words. Thus we think it natural to ask the following, which was our motivating question from the introduction:

\begin{ques}\label{ques1}
Is there an algorithm that, on input of a finite presentation $P=\left \langle X|R \right \rangle$ of a non-trivial group $\overline{P}$, outputs a word $w \in W(X)$ such that $w$ is non-trivial in $\overline{P}$?
\end{ques}

We conjecture that the answer to this question is no, on the basis of the results we have already obtained in this section. As a corollary to theorem \ref{unbounded word}, we show the following closely related result.

\begin{cor}\label{ab bounded}
Fix any $k>0$. Then there is no algorithm that, on input of a finite presentation $P=\left \langle X|R \right \rangle$ of a non-abelian group $\overline{P}$ outputs two words $w, w'\in W(X)$, each of length at most $k$, such that $[w, w']$ is non-trivial in $\overline{P}$.
\end{cor}

\begin{proof}
Assume such an algorithm exists; we use this to contradict theorem \ref{unbounded word}. Given any finite presentation $P=\left \langle X|R \right \rangle$ of a non-trivial group, recursively form the abelianisation presentation $P^{\ab}$ and recursively test if $\overline{P^{\ab}}$ is trivial. If not, then recursively test all the generators from $P^{\ab}$ to find a non-trivial one; the corresponding generator in $\overline{P}$, call this $w$, must also be non-trivial, and has length $1$. If however $\overline{P^{\ab}}$ is trivial, then $\overline{P}$ must be perfect and hence non-abelian, since it is non-trivial. So use the algorithm to output words $w,w'\in W(X)$, each of length at most $k$, such that $[w,w']$ is non-trivial in $\overline{P}$. 
In either case we have recursively constructed a word $w \in W(X)$ that has length at most $k$ and is non-trivial in $\overline{P}$, contradicting theorem \ref{unbounded word}.
\end{proof}

\begin{cor}
There is no algorithm that, on input of a finite presentation $P=\left \langle X|R \right \rangle$ of a non-abelian group $\overline{P}$, outputs two generators $x,y \in X$  such that $[x,y]$ is non-trivial in $\overline{P}$.
\end{cor}

Again, we are tempted to generalise this by removing the bound on the length of the output words, as stated in the following question:

\begin{ques}\label{ques2}
Is there an algorithm that, on input of a finite presentation $P=\left \langle X|R \right \rangle$ of a non-abelian group $\overline{P}$, outputs two words $w, w' \in W(X)$ such that $[w, w']$ is non-trivial in $\overline{P}$.
\end{ques}

We conjecture that the answer to this question is no. It can be easily shown by an argument almost identical to the proof of corollary \ref{ab bounded} that `no' to question \ref{ques2} implies `no' to question \ref{ques1}. Our partial results on these two questions have used very similar methods of proof, and it can be seen that the two questions are very closely and naturally related.

\section{Finding splittings}

We now proceed to our results regarding the splitting of groups into free products. As we shall see later, our results have applications to existing constructions in splittings of groups with more than one end, as well as splittings of manifolds.


The following two theorems can be found in \cite{MKS}, on p.~192 and p.~245 respectively. Using these we demonstrate an important property of the groups constructed in theorem \ref{chuck result}.

\begin{thm}[Grushko-Neumann theorem]\label{Grushko}
Let $G$ be a finitely generated group, and suppose that $G \cong A*B$. Then $\rank(G)=\rank(A)+\rank(B)$.
\end{thm}






\begin{thm}[Grushko-Neumann decomposition]\label{grush neum}
Let $P$ be a finite presentation of a group that splits as a free product $A_{1}* \dots *A_{n}$, with all the $A_{i}$ indecomposable. 
Let $B_{1}* \dots *B_{k}$ be another such splitting into indecomposable 
groups. Then $n=k$, and there exists a permutation $\sigma \in S_{n}$ such that $A_{i } \cong B_{\sigma(i)}$ for all $1 \leq i \leq n$.
\end{thm}

\begin{lem}\label{dont split}
If the group $\overline{\Pi}_{m,n}$ from theorem \ref{chuck result} is non-trivial (equivalently, $n \notin W_{m}$), then it is indecomposable.
\end{lem}

\begin{proof}
Assume it is not; we proceed by contradiction. If $\overline{\Pi}_{m,n}$ is non-trivial then by corollary \ref{rabin lem} it is perfect and has rank precisely 2. However, the Grushko-Neumann theorem asserts that a non-trivial splitting of it must have rank 2, so would be the free product of 2 groups of rank 1; two cyclic groups. This would have non-trivial abelianisation, namely the direct sum of two non-trivial cyclic groups, contradicting the fact that $\overline{\Pi}_{m,n}$ is perfect.
\end{proof}

We can now prove the main technical result for this section; theorem \ref{split} follows as an immediate consequence.

\begin{lem}\label{presplit}
There is no algorithm that, on input of a finite presentation of the form $Q:=\Pi_{n,a}*\Pi_{n,b}* \Pi_{n,c}$ where $|\{a, b, c\} \cap W_{n}| \leq 1$, outputs two finite presentations $P_{1}, P_{2}$ which represent non-trivial groups and whose free product $\overline{P}_{1}* \overline{P}_{2}$ is isomorphic to $\overline{Q}$.
\end{lem}

Note that the above group $\overline{Q}$ will always split as a non-trivial free product, as at least two of the groups $\overline{\Pi}_{n,a}, \overline{\Pi}_{n,b}, \overline{\Pi}_{n,c}$ must be non-trivial since $|\{a, b, c\} \cap W_{n}| \leq 1$.

\begin{proof}
Assume such an algorithm exists; we proceed to contradict lemma \ref{two groups recursion}. 
So given $n, a, b, c$ with $|\{a, b, c\} \cap W_{n}| \leq 1$, we can construct the presentation $Q:=\Pi_{n,a}*\Pi_{n,b}* \Pi_{n,c}$, and use the algorithm to split this as the free product of two non-trivial groups with presentations $P_{1}, P_{2}$.
There are two possible cases to consider:

Case 1. Precisely one of $a,b,c$ lies in $W_{n}$. If $a \in W_{n}$, then $\overline{Q} \cong \overline{\Pi}_{n,b}* \overline{\Pi}_{n,c}$, and both of these two factors are non-trivial by theorem \ref{chuck result}, and indecomposable by lemma \ref{dont split}. So $\overline{\Pi}_{n,b}* \overline{\Pi}_{n,c} \cong \overline{P}_{1} * \overline{P}_{2}$. Since the splitting on the left is into indecomposable groups, then theorem \ref{grush neum} asserts that so too is the splitting on the right, and moreover that at least one of $\overline{\Pi}_{n,b}, \overline{\Pi}_{n,c}$ is isomorphic to $\overline{P}_{1}$. If instead $b \in W_{n}$ (resp. $c \in W_{n}$), then an identical argument would show that at least one of $\overline{\Pi}_{n,a}, \overline{\Pi}_{n,c}$ (resp. $\overline{\Pi}_{n,a}, \overline{\Pi}_{n,b}$) is isomorphic to $\overline{P}_{1}$. So regardless of which of $a,b,c$ lies in $W_{n}$, we can conclude that at least one of $\overline{\Pi}_{n,a}, \overline{\Pi}_{n,b}, \overline{\Pi}_{n,c}$ is isomorphic to $\overline{P}_{1}$.

Case 2. None of $a,b,c$ lie in $W_{n}$. Then $\overline{Q} \cong \overline{\Pi}_{n,a}* \overline{\Pi}_{n,b}* \overline{\Pi}_{n,c}$ is a splitting into indecomposable groups. So precisely one of $\overline{P}_{1}, \overline{P}_{2}$ splits as a free product since $\overline{\Pi}_{n,a}* \overline{\Pi}_{n,b}* \overline{\Pi}_{n,c} \cong \overline{P}_{1}*\overline{P}_{2}$ (using theorem \ref{grush neum} and observing that the left hand side in this expression is split into indecomposable groups). So at least one of $\overline{\Pi}_{n,a}, \overline{\Pi}_{n,b}, \overline{\Pi}_{n,c}$ is isomorphic to at least one of $\overline{P}_{1}$ or $\overline{P}_{2}$.
\\In either case, at least one of $\overline{\Pi}_{n,a}, \overline{\Pi}_{n,b}, \overline{\Pi}_{n,c}$ is isomorphic to at least one of $\overline{P}_{1},\overline{P}_{2}$; these latter two being non-trivial groups. Lemma \ref{isomorphism} allows us to recursively search for an isomorphism between one of $\overline{\Pi}_{n,a}, \overline{\Pi}_{n,b}, \overline{\Pi}_{n,c}$ and one of $\overline{P}_{1},\overline{P}_{2}$, which will eventually halt and yield an isomorphic pair. This will identify one of the $\overline{\Pi}_{n,a},\overline{\Pi}_{n,b}, \overline{\Pi}_{n,c}$ as non-trivial, and hence one of $a,b,c$ not in $W_{n}$. As all the above steps have been algorithmic, we thus have an algorithm that on input of a 4-tuple $(n, a,b,c)$ with $|\{a, b,c\} \cap W_{n}| \leq 1$, outputs $z \in \{a,b,c\}$ such that $z \notin  W_{n}$. But this contradicts lemma \ref{two groups recursion}, so our original algorithm cannot possibly exist.
\end{proof}

\begin{thm}\label{split}
There is no algorithm that, on input of a finite presentation $P$ of a group that is a free product of two non-trivial finitely presented groups, outputs two finite presentations $P_{1}, P_{2}$ which represent non-trivial groups and whose free product is isomorphic to $\overline{P}$.
\end{thm}

\section{Finding embeddings}\label{dunno}

Our final results relate to finding explicit embeddings of finitely presented groups, and finding finitely presented subgroups of finitely presented groups. From hereon, $|g|$ denotes the order of a group element $g$; we say $g$ is \emph{torsion} if $1<|g|<\infty$.

\begin{defn}
Let $G$ be any group. We define the set $\torord(G):=\{n \in \mathbb{N}\ |\ \exists g \in G \textnormal{ with } |g|=n\geq 2\}$, the set of orders of non-trivial torsion elements of $G$. Note that $\torord(G)$ never contains $0$ or $1$.
\end{defn}

The following theorem is from \cite{Rot} theorem 11.69.

\begin{thm}[Torsion theorem for amalgamated products and HNN extensions]\mbox{}\label{tor thm}
\\Let $g\in G$ have finite order in $G$. Then:
\\1. If $G=K_{1} *_{H}K_{2}$ is an amalgamated product, then $g$ is conjugate to an element of $K_{1}$ or $K_{2}$. Hence $\torord(K_{1}*_{H}K_{2})=\torord(K_{1})\cup\torord(K_{2})$.
\\2. If $G=K*_{H}$ is an HNN extension, then $g$ lies in the base group $K$. Hence $\torord(K*_{H})=\torord(K)$.
\end{thm}


We note here that in the proof of theorem \ref{rabin} found in \cite{Gor}, $P(w)$ is built up from amalgamated products and HNN extensions, beginning with the finite  presentation $P$ and a free group. Thus we have the following immediate corollary.

\begin{cor}\label{torord2} Let $P$ be a finite presentation of a group, and $w$ a word on its generators representing a non-trivial element of $\overline{P}$. Then $\torord\left(\overline{P(w)}\right)=\torord\left(\overline{P}\right)$.
\end{cor}

The following is a simple construction based on the above corollary; we later use it to prove further results.

\begin{defn}\label{psi}
Let $p_{i}$ be the $i^{th}$ prime (starting at $p_{0}=2$), and $x_{i}$ one of the two generators of $\Pi_{n,i}$. Then  for any $n,i \in \mathbb{N}$ we define the finite presentations $\Psi_{n,i}:=(\Pi_{n,i}\times C_{p_{i}})(x_{i})$ and $\Phi_{n,i}:=(\Pi_{n,i}\times C_{2})(x_{i})$, as described in theorem \ref{rabin}. Hence if $i \notin W_{n}$ then $\torord(\overline{\Psi}_{n,i})= \{p_{i}\}$ and $\torord(\overline{\Phi}_{n,i})= \{2\}$; if $i \in W_{n}$ then $\torord(\overline{\Psi}_{n,i})= \torord(\overline{\Phi}_{n,i})= \emptyset$ as they are both trivial.
\end{defn}

%


Using the above observations as well as our previous results in recursion theory, we now prove the main technical result for this section.

\begin{lem}\label{tor elt}
There is no algorithm that, on input of a finite presentation $P=\langle X|R \rangle$ of a group with torsion, and some $n \in \torord(\overline{P})$, outputs a word $w \in W(X)$ which represents any torsion element of $\overline{P}$ (not necessarily of order $n$).
\end{lem}

\begin{proof}
Suppose such an algorithm exists; we proceed to contradict lemma \ref{two groups recursion}. So take $n,a,b$ with $|\{a,b\} \cap W_{n}| \leq 1$. Now construct the presentations $\Phi_{n,a}, \Phi_{n,b}$ as in definition \ref{psi}. Finally, form the free product presentation $\Phi_{n,a} * \Phi_{n,b}$ and call this $P$.
Now, since at least one of $a,b\notin W_{n}$, we must have that at least one of $\overline{\Pi}_{n,a}, \overline{\Pi}_{n,b}$ is non-trivial (say $\overline{\Pi}_{n,i}$, where $i \in \{a,b\}$). Hence $x_{i}$ is non-trivial in $\overline{\Pi}_{n,i}$ by corollary \ref{rabin lem}, and so $\overline{\Phi}_{n,i}$ is non-trivial by theorem \ref{rabin}.
%
%
So by the properties of HNN extensions and amalgamated products, along with the construction from theorem \ref{rabin}, we have that $C_{2}\hookrightarrow \overline{\Pi}_{n,i}\times C_{2} \hookrightarrow \overline{(\Pi_{n,i}\times C_{2})(x_{i})}=\overline{\Phi}_{n,i} \hookrightarrow \overline{\Phi}_{n,a}* \overline{\Phi}_{n,b}$. Now knowing that $\overline{\Phi}_{n,a}* \overline{\Phi}_{n,b}$ has an element of order $2$, we can apply our algorithm to produce a word $w$ representing a torsion element in $\overline{\Phi}_{n,a}* \overline{\Phi}_{n,b}$.
%
By theorem \ref{tor thm} we have that $w$ must be conjugate to an element in either $\overline{\Phi}_{n,a}$ or $\overline{\Phi}_{n,b}$. We can recursively search for this, and find one of these which $w$ is conjugated into. This gives us one of $\overline{\Phi}_{n,a}$, $\overline{\Phi}_{n,b}$ which is non-trivial, hence one of $x_{a}, x_{b}$ which is non-trivial, hence one of $\overline{\Pi}_{n,a}$, $\overline{\Pi}_{n,b}$ which is non-trivial, hence one of $a,b \notin W_{n}$. All stages in this process have been recursive, hence we have contradicted lemma \ref{two groups recursion}.
\end{proof}

It is a straightforward application of this technical result to prove the following theorem.

\begin{thm}\label{no embeddings}
There is no algorithm that, on input of two finite presentations $P=\langle X|R\rangle$ and $Q=\langle Y |S \rangle$ such that $\overline{P}$ embeds in $\overline{Q}$, outputs an explicit map $\theta: X \to W(Y)$ such that $\theta$ extends to an embedding $\overline{\theta}: \overline{P} \hookrightarrow \overline{Q}$.
\end{thm}

\begin{proof}
Assume we have such an algorithm; we proceed to contradict lemma \ref{tor elt}. So take a finite presentation $Q=\langle Y |S \rangle$ of a group with torsion, and some $n \in \torord(\overline{Q})$. Then the group $C_{n}$ with finite presentation $\left \langle t | t^{n} \right \rangle$ must embed in $\overline{Q}$. Now we apply our algorithm to construct a map $\theta: \{t\} \to W(Y)$ which extends to an embedding $\overline{\theta}: C_{n} \hookrightarrow \overline{Q}$. Taking $w:= \theta(t)$ we have that $|w|=n$, so $w$ represents a torsion element of $\overline{Q}$. All stages in this process have been recursive, hence we have algorithmically constructed a word representing a torsion element of $\overline{Q}$, contradicting lemma \ref{tor elt}.
\end{proof}



Remarkably, even if we specify an order $n$, we still can't find such a torsion element $w$.
Not only do we lack an algorithm to output torsion elements; we can't even say much about $\torord(G)$ for an arbitrary group $G$, as the following shows.

\begin{thm}\label{tor n}
There is no algorithm that, on input of a finite presentation $P=\langle X|R \rangle$ of a group with torsion, outputs some $n \in \torord(\overline{P})$.
\end{thm}

\begin{proof}
We again proceed to contradict lemma \ref{two groups recursion}. So take $n,a,b$ with $|\{a,b\} \cap W_{n}| \leq 1$, and recursively construct the presentations $\Psi_{n,a}, \Psi_{n,b}$ as in definition \ref{psi}. Then $\torord(\overline{\Psi}_{n,a})$ is either $\emptyset$ or $\{p_{a}\}$, and $\torord(\overline{\Psi}_{n,b})$ is either $\emptyset$ or $\{p_{b}\}$.
So by theorem \ref{tor thm}, $\torord(\overline{\Psi}_{n,a}*\overline{\Psi}_{n,b}) \subset\{p_{a}, p_{b}\}$. Moreover, $p_{a}\in \torord(\overline{\Psi}_{n,a}*\overline{\Psi}_{n,b})$ if and only if $a \notin W_{n}$ (and similarly for $p_{b}$); $\torord(\overline{\Psi}_{n,a}*\overline{\Psi}_{n,b})$ is non-empty since at least one of $a,b \notin W_{n}$. Now if we could algorithmically output one of $p_{a}, p_{b} \in \torord(\overline{\Psi}_{n,a}*\overline{\Psi}_{n,b})$, we would then have one of $a, b \notin W_{n}$, contradicting lemma \ref{two groups recursion}.
\end{proof}

At this point is seems relevant to make the following remark about the enumerability of finitely presented subgroups of finitely presented groups.

\begin{thm}\label{fg subgroups}
There exists a finitely presented group $G$ such that the set of all finite presentations that define groups which embed into $G $ is not recursively enumerable.
\end{thm}

 We break up the proof into preparatory lemmas, beginning with the following two results found in \cite{Rot} as corollary 11.72 and theorem 12.18 respectively.

\begin{lem}\label{2 gen tor}
Let $G$ be a countable group with generator and relator sets that are recursively enumerable. Then $G$ can be uniformly embedded into some 2-generator recursively presented group $E$ such that $\torord (G)=\torord(E)$.
\end{lem}


\begin{thm}[Higman]\label{higman}
Let $G$ be a finitely generated recursively presented group. Then $G$ can be uniformly embedded into some finitely presented group $H$ such that $\torord(G)=\torord(H)$.
\end{thm}

In both cases, the group so constructed is built up from amalgamated products and HNN extensions, beginning with $G$ and some finitely generated free groups. Hence, by theorem \ref{tor thm}, $\torord(G)=\torord(E)$ (respectively, $\torord(H)$).

\begin{lem}\label{no re tor}
There is a uniform procedure than, on input of any $n \in \mathbb{N}$, constructs a finite presentation $Q_{n}$ such that $\torord(\overline{Q}_{n})$ is one-one equivalent to $\mathbb{N}\setminus W_{n}$. Taking $n'$ with $W_{n'}$ non-recursive thus gives that $\torord(\overline{Q}_{n'})$ is not recursively enumerable.
\end{lem}

See \cite{Rogers} section $\S 7.1$ for an introduction to one-one equivalent sets. We wish to thank Chuck Miller for pointing out a more elegant presentation for the group $\overline{P}_{n}$ used in the following proof.

\begin{proof}
For any given $n \in \mathbb{N}$, we may form the presentation $P_{n}:=\langle x_{0}, x_{1}, \ldots | x_{i}^{p_{i}}=e\ \forall i \geq 0, x_{j}=e\ \forall j \in W_{n} \rangle$ with recursively enumerable generating and relating sets; by construction $\torord(\overline{P}_{n})=\{p_{j}| j \notin W_{n}\}$.
We use lemma \ref{2 gen tor} to uniformly construct a 2-generator recursive presentation $H_{n}$ with $\overline{P}_{n}\hookrightarrow \overline{H}_{n}$ and $\torord(\overline{H}_{n})=\torord(\overline{P}_{n})$, and theorem \ref{higman} to uniformly construct a finite presentation $Q_{n}$ with $\overline{H}_{n}\hookrightarrow \overline{Q}_{n}$ and $\torord(\overline{Q}_{n})=\torord(\overline{H}_{n})=\torord(\overline{P}_{n})$. Since $p_{j} \in \torord(\overline{Q}_{n})$ if and only if $j \notin W_{n}$, we have that $\torord(\overline{Q}_{n})$ is one-one equivalent to $\mathbb{N} \setminus W_{n}$. And as all stages in this construction have been uniform, we conclude that such a presentation $Q_{n}$ can be uniformly constructed from $n$.
\end{proof}




\begin{proof}[Proof of theorem \ref{fg subgroups}]
We make the simple observation that $\overline{\langle t|t^{m} \rangle} $ embeds in a group $G$ if an only if $m \in \torord(G)$. Now take $G$ to be $\overline{Q}_{n'}$ from lemma \ref{no re tor}, with $W_{n'}$ non-recursive. Begin an enumeration of the finite presentations which embed in $\overline{Q}_{n'}$ and look for those of the form $\langle t|t^{m} \rangle $, which will in turn give a recursive enumeration of $\torord(\overline{Q}_{n'})$. This is impossible by lemma \ref{no re tor}.
\end{proof}

\section{Applications}

We now show an application that relates to the following result by Stallings (\cite{Stallings} Theorem 4.A.6.5 and Theorem 5.A.9) on splitting groups with more than one end. For an introduction to the ends of a group, see \cite{Scott Wall}.

\begin{thm}[Stallings]\label{Stal}A finitely generated group $G$ has more than one end if and only if it splits over a finite subgroup via an amalgamated product or HNN extension.
\end{thm}

\begin{cor}
There is no algorithm that, on input of a finite presentation $P$ of a group $\overline{P}$ with more than one end, splits $P$ over a finite subgroup. That is, re-writes the presentation $P$ as an HNN extension or amalgamated product over a finite subgroup.
\end{cor}

\begin{proof}
Whenever $|\{a, b, c\} \cap W_{n}| \leq 1$, we have that at least two of $\overline{\Pi}_{n,a}, \overline{\Pi}_{n,b}, \overline{\Pi}_{n,c}$ are non-trivial by theorem \ref{chuck result}, and all three are torsion free and indecomposable. So $\overline{\Pi}_{n,a}* \overline{\Pi}_{n,b}* \overline{\Pi}_{n,c}$ splits, but not over a non-trivial finite subgroup. Also, since the non-trivial $\overline{\Pi}_{n,-}$ groups are perfect and hence non-cyclic, their free product cannot be an HNN extension over the trivial group (equivalently, a free product with $\mathbb{Z}$). So the only way to split $\overline{\Pi}_{n,a}* \overline{\Pi}_{n,b}* \overline{\Pi}_{n,c}$ over a finite subgroup is to split it as a free product. But lemma \ref{presplit} asserts that this cannot always be done algorithmically.
\end{proof}

We finish with an application involving a proof by Markov (\cite{Markov}) that there is no algorithmic classification of closed $4$-manifolds. 
We write the connect sum of two manifolds $M$ and $N$ as $M \# N$, and if they are homeomorphic we denote this by $M\simeq N$. For the remainder of the geometry concepts used below, we recommend the reader consult a reference such as \cite {Hatcher}.

\begin{thm}[Markov]\label{markov}
There is a recursive procedure that, on input of a finite presentation $P=\left \langle X|R \right \rangle$ of a group, constructs a finite triangulation $M(P)$ of a closed $4$-manifold with the following properties:
\\1. $\pi_{1}(M(P))\cong \overline{P}$.
\\2. If $P$ and $Q$ are finite presentations, then $M(P*Q) \simeq M(P)\# M(Q)$.
\end{thm}

We wish to thank Jacob Rasmussen for his careful explanation of the following construction (originally due to Markov).

\begin{proof}
Property 1 is explicitly stated, and a construction given, in \cite{Markov}. However, a more detailed exposition can be found in \cite{BHP}. It remains only to prove that the construction satisfies property 2. 
We loosely describe the construction given in \cite{Markov}: Start with the $4$-ball $B^{4}$. For each generator of the finite presentation $P$, attach a 1-handle to $B^{4}$. For each relator, attach a 2-handle in some prescribed way. This can be done in such a way that the resulting space is a 4-manifold, and by taking the double of this space we obtain our closed $4$-manifold having fundamental group isomorphic to $\overline{P}$. If we choose a separating 3-ball $D$ through our original ball $B^{4}$, then by starting with a presentation $P*Q$ one can observe that all the 1-handles corresponding to generators from $P$ can be attached to one side of $D$, and those from $Q$ can all be attached to the other side. Also, the 2-handles can be attached in such a way that those induced by $P$ remain on the same side of $D$ as all the 1-handles from $P$, and likewise for $Q$. Now, when we take the double of the resulting space, the two copies of $D$ are identified along their boundaries, thus forming a separating sphere $S^{3}$. Hence our resulting space splits as the connect sum of two spaces; it is not hard to see that these two are just the spaces obtained by applying the procedure to $P$ and $Q$.
\end{proof}


\begin{cor}
There is no algorithm that, on input of a finite triangulation of a closed $4$-manifold $M$ which splits as a connect sum of two non-simply connected manifolds, outputs two finite triangulations of non-simply connected closed $4$-manifolds $M_{1}, M_{2}$ whose connect sum is homeomorphic to $M$.
\end{cor}


\begin{proof}
Assume such an algorithm exists; we proceed to contradict lemma \ref{presplit}. So given $n,a,b,c$ with $|\{a, b, c\} \cap W_{n}| \leq 1$, we construct $M:=M(\Pi_{n,a}* \Pi_{n,b}* \Pi_{n,c})$ as in theorem \ref{markov}. We know this splits as a connect sum of non-simply connected closed $4$-manifolds by property 2 and the fact that at least two of $\overline{\Pi}_{n,a}, \overline{\Pi}_{n,b}, \overline{\Pi}_{n,c}$ are non-trivial. Applying our algorithm to $M$ gives two finite triangulations of non-simply connected closed $4$-manifolds $M_{1}$ and $M_{2}$ such that $M_{1}\# M_{2} \simeq M$. 
From the finite triangulation of $M_{1}$ we can algorithmically construct a finite presentation $A$ of $\pi_{1}(M_{1})$ (see \cite{Seifert} $\S 46$), and likewise a finite presentation $B$ of $\pi_{1}(M_{2})$. 
By the Seifert-van Kampen theorem (\cite{Hatcher}) and theorem \ref{markov} we have $\overline{\Pi}_{n,a}* \overline{\Pi}_{n,b}* \overline{\Pi}_{n,c} \cong \pi_{1}(M) \cong \pi_{1}(M_{1}) * \pi_{1}(M_{2}) \cong \overline{A}*\overline{B}$, with $\overline{A}$ and $\overline{B}$ both non-trivial. Since all the above steps have been uniform, we have thus algorithmically constructed a non-trivial free product splitting of $\overline{\Pi}_{n,a}* \overline{\Pi}_{n,b}* \overline{\Pi}_{n,c}$, contradicting lemma \ref{presplit}.
\end{proof}

We can also apply the Markov construction to carry over some of our other algorithmic results from algebra to geometry.

\begin{cor}
There is no algorithm that, on input of a finite triangulation of a closed $4$-manifold $M$ such that $\pi_{1}(M)$ has torsion, outputs an essential loop $\gamma$ in $M$ which represents a torsion element in $\pi_{1}(M)$. Nor is there an algorithm that, on the same input as above, outputs some $n \geq 2$ such that there exists an essential loop $\beta$ in $M$ with $[\beta]$ having order precisely $n$ in $\pi_{1}(M)$.
\end{cor}

\begin{proof}
Given a finite presentation $P$ such that $\overline{P}$ has torsion, we may use theorem \ref{markov} to uniformly construct $M(P)$ with $\pi_{1}(M(P)) \cong \overline{P}$. Existence of the first algorithm would thus contradict lemma \ref{tor elt}; existence of the second, theorem \ref{tor n}.
\end{proof}

\begin{rem}
By looking closely at the proof of lemma \ref{tor elt} and applying it to the Markov construction, we see that we have actually shown a stronger result: There is no algorithm that, on input of a finite triangulation of a closed $4$-manifold $M$ such that $\pi_{1}(M)$ has torsion, and 2 loops $\gamma, \beta$ such that at least one represents an element of order precisely 2 in $\pi_{1}(M)$, outputs any essential loop $\alpha$ in $M$ which represents a torsion element of any order in $\pi_{1}(M)$.
\end{rem}


\ \\
\\Maurice Chiodo
\\Department of Mathematics and Statistics
\\The University of Melbourne
\\Parkville, VIC, 3010
\\AUSTRALIA
\\\textit{m.chiodo@pgrad.unimelb.edu.au}

\end{document}